\documentclass[12 pt]{amsart}
\usepackage{amsmath,times,epsfig,amssymb,amsbsy,amscd,amsfonts,amstext,color,bm}
\usepackage[latin1]{inputenc}
\usepackage{graphics, graphicx, xcolor}
\usepackage{amscd}
\usepackage[all]{xy}

\newcommand{\C}{\mathbb{C}}

\newcommand{\R}{\mathbb{R}}
\newcommand{\Lbb}{\mathbb{L}}

\newcommand{\Z}{\mathbb{Z}}
\newcommand{\Q}{\mathbb{Q}}

\newcommand{\PP}{\mathcal{P}}

\renewcommand{\to}{\longrightarrow}

\newtheorem{Theorem}{Theorem}[section]
\newtheorem{Definition}[Theorem]{Definition}

\newtheorem{Lemma}[Theorem]{Lemma}
\newtheorem{Proposition}[Theorem]{Proposition}
\newtheorem{Corollary}[Theorem]{Corollary}
\newtheorem{Remark}[Theorem]{Remark}
\newtheorem{Example}[Theorem]{Example}

\newtheorem{Notation}[Theorem]{Notation }

\renewcommand{\ni}[0]{\noindent}
\newcommand{\Tbf}[0]{\mathbf{T}}

\DeclareMathOperator{\codim}{codim}

\DeclareMathOperator{\Cone}{Cone}

\newcommand{\LL}{\mathcal{L}}

\newcommand{\ch}{\mathsf{ch}}

\DeclareMathOperator{\rank}{rank}

\renewcommand{\AA}{\mathcal{A}}
\newcommand{\FF}{\mathcal{F}}
\renewcommand{\SS}{\mathcal{S}}
\renewcommand{\LL}{\mathcal{L}}
\newcommand{\MM}{\mathcal{M}}
\renewcommand{\AA}{\mathcal{A}}

\newcommand{\Cal}{\mathcal }

\newcommand{\pol}{\lhd} 
\newcommand{\pog}{\rhd} 
\newcommand{\fl}{\prec} 
\newcommand{\fg}{\succ} 

\newcommand{\op}{\textup{op}}

\newcommand{\nbc}{\textit{\textbf{nbc}}}

\theoremstyle{remark}

\theoremstyle{definition}

\newcommand{\totl}{\sqsubset}

\usepackage{color}
\usepackage{ifthen}

\newboolean{usecolor}
\setboolean{usecolor}{true}

\ifthenelse{\boolean{usecolor}}
{
  \definecolor{colore}{cmyk}{0,1,0.6,0}
  \definecolor{coloregen}{cmyk}{0.7,0,1,0}
  \definecolor{coloresimo}{cmyk}{1,0.6,0,0}
}
{
  \definecolor{colore}{cmyk}{0,0,0,1}
  \definecolor{coloregen}{cmyk}{0,0,0,1}
  \definecolor{coloresimo}{cmyk}{0,0,0,1}
}

\begin{document}

\title{The $\nbc$ minimal complex of supersolvable arrangements}

\begin{abstract}
In this paper we give a very natural description of the bijections between the minimal CW-complex homotopy equivalent to the complement of a supersolvable arrangement $\AA$, the $\nbc$ basis of the Orlik-Solomon algebra associated to $\AA$ and the set of chambers of $\AA$. We use these bijections to get results on the first (co)homology group of the Milnor fiber of $\AA$ and to describe a bijection between the symmetric group and the $\nbc$ basis of the braid arrangement.

\end{abstract}

\author{Simona Settepanella}
\address{Simona Settepanella, Department of Mathematics, Hokkaido University, Kita 10, Nishi 8, Kita-Ku, Sapporo 060-0810, Japan.}
\email{settepanella@math.sci.hokudai.ac.jp}

\author{Michele Torielli}
\address{Michele Torielli, Department of Mathematics ``Giuseppe Peano'', Turin University,  Via Carlo Alberto 10, 10123 Turin, Italy.}
\email{michele.torielli@unito.it}


\date{\today}
\maketitle

\tableofcontents

\section{Introduction}

The theory of arrangements of hyperplanes is a subject intensively studied during the last 60 years. The main topic of this theory is the study of the complement of a set of hyperplanes in the space. It started in 1889  when  Roberts gave a formula to count how many open disconnected regions there are when we cut the plane by removing a set of lines (see \cite{orlterao} for a detailed reference). A direct generalization of this problem, the removal of hyperplanes in higher dimensional spaces, stayed  unsolved until 1975, when Zaslavsky gave a general counting formula in a paper on the AMS Memoir (see \cite{Zaslavsky1975Memoir}). Those open regions are called \textit{chambers}. In 1980 Orlik and Solomon introduced the well known Orlik-Solomon algebra (see \cite{orlik1980combinatorics}) that is completely described by combinatorial methods and compute the cohomology group with integer coefficients of the complement of a complex hyperplanes arrangement. 

The Orlik-Solomon algebra is a graded algebra with a basis $\nbc$ called \textit{non broken circuit} basis. It turns out that, when considering complexified real arrangements, i.e. the case in which the hyperplanes have real defining equations, the total number of elements in a non broken circuit basis equals the number of chambers of the underling real arrangement. The correspondence between those two objects has been studied by many authors interested in the combinatorial aspects of the theory of arrangements of hyperplanes. For example, Barcelo and Gupil (see \cite{BarceloGupil}) studied the case of arrangements coming from reflection groups and,  Gioan and Las Vergnas, in \cite{GioLa}, studied the general case.

More recently, Dimca and Papadima (see \cite{dimca2003hypersurface}), and Randell (see \cite{randell}) proved that the complement of a complex hyperplane arrangement is a minimal space, i.e. it has the homotopy type of a CW-complex with exactly as many $k$-cells as the $k$-th Betti number $b_k$ or, in other words, as many $k$-cells as the cardinality of $\nbc_k$, i.e. of the homogeneous elements of degree $k$ in the non broken circuits  basis.
In 2007, Yoshinaga (see \cite{yoshinaga2007hyperplane}), Salvetti and the first author (see \cite{salvettisette}) gave a description of this minimal complex in the case of complexified real arrangements.  Then the question arises on existence of ``natural" bijections between the chambers of the real arrangement, the minimal CW-complex of the complexified one and the $\nbc$ basis. This question has been addressed by Delucchi in \cite{Del2008} and Yoshinaga in  \cite{yoshinaga2009chamber}.

In this paper, we study this problem in the special case of supersolvable arrangements defining a very natural and handy description of the bijection between minimal complex, $\nbc$ basis and chambers. This bijection, defined in Section  \ref{sec:main}, only involves elements of the intersection poset of the hyperplane arrangement endowed with an order $\pol$ and the hyperplanes separating chambers and a previously fixed chamber $C_0$.

The map $f$ defined in equation (\ref{df:nabla}), Section \ref{sec:main}, turns out to have applications to the study of first (co)homology group of the Milnor fiber of supersolvable arrangements via the study of the first (co)homology group  with local coefficients of their complement. The (co)homolgy of the Milnor fiber associated to an arrangement has been studied by several authors (see, for instance, \cite{budur2009first}, \cite{cohen1995milnor}, \cite{denham2012multinets}, \cite{libgober2010combinatorial}, \cite{yoshinaga2013milnor} and the survey \cite{Suciu:2013fk}) as the first (co)homology group  with local coefficients of the complement of an arrangement (see, for instance, \cite{libgober2000cohomology}, \cite{dimca2009admissible}, \cite{schechtman1994local} and \cite{ontheadmisscertlocsyst}).

By means of the map $f$ we construct a filtration of the $\nbc$ minimal complex of a supersolvable arrangement that allows to apply similar method to the one used in \cite{DeProSal} and developed in \cite{settepanella2009cohomology}. In particular, if $F(\AA)$ is the Milnor fiber of $\AA$,   then the following statement holds.
  
\begin{Theorem}\label{th:bounintro} Let $\AA =\AA_d \supset \ldots \supset \AA_1$ be a supersolvable arrangement in $\R^d$. If it exists an index  $1 \leq j \leq d$ such that the cardinalities of $\AA_{j-1}$ and $\AA_j$ are coprime and all hyperplanes in $\AA_j \setminus \AA_{j-1}$ intersect generically the hyperplane in $\AA_1$, then  $H_1(F(\AA_j),\Q)\simeq \Q^{b_j}$, $b_j \leq \sharp (\AA_j \setminus \AA_{j-1})$.
\end{Theorem}

The Theorem above proves that, under certain conditions on a supersolvable arrangement $\AA$, we get the triviality of the monodromy action on $H_1(F(\AA),\Q)$. A natural question is under which conditions Theorem \ref{th:bounintro} can be extended to higher (co)homology groups and which informations we can get on the monodromy action on $H_1(F(\AA),\Q)$ ( and hence on $H_1(F(\AA),\C)$ ) using the filtration described in Section \ref{sec:mil}.

Moreover, if  $\AA=\{H_{ij}=\{x_i=x_j\},1\le i<j\le n+1\}$ is the braid arrangement, the map $f$ turns out to be very useful in order to give, for the first time, a direct description of the bijection between the symmetric group and the non broken circuit basis $\nbc$ associated to $\AA$. In 1995 Barcelo and Goupil (see \cite{BarceloGupil}) proved that if $\AA(W)$ is the reflection arrangement associated to the Coxeter group $(W,S)$, $\nbc (\AA(W))$ is its non broken circuit basis and $H_{r} \in \AA(W)$ is the hyperplane defined by the reflection $r \in W$, the map
\begin{equation*}
\begin{split}
g\colon  \nbc (\AA(W)) &\to W\\
 (H_{r_1}, \dots, H_{r_k}) &\longmapsto w=r_1 \dots r_k
\end{split}
\end{equation*}
is a bijection. However, they could not provide a direct description of the inverse map $g^{-1}$ because it is related to the word problem in the group $W$. Indeed the expression of $w \in W$ as a product of reflections is not unique. In this paper we provide a description of $g^{-1}$ in case of the symmetric group on $n+1$ elements $A_n$.  Since the map in \cite{BarceloGupil} is for any reflection group it is a natural question whether the construction in this paper can be extended also to other reflection groups (even the non-supersolvable ones).

More in detail, in Section \ref{sec:pre} we recall the definitions of the minimal Salvetti's complex associated to complexified real arrangements and of the $\nbc$ basis for real supersolvable arrangements. In Section \ref{sec:main} we introduce a relation between the $\nbc$ basis and the minimal complex of a complexified real supersolvable arrangement $\AA$ and prove that this relation is, in fact, a bijection. In Section \ref{sec:mil} and in Section \ref{sec:braid} we consider two applications of the map described in Section \ref{sec:main}. In Section \ref{sec:mil} we use the map to get result on the first (co)homology group of the Milnor fiber of the arrangement $\AA$, while in Section \ref{sec:braid} we give a description of this map in the special case of the braid arrangement providing a bijection between elements of the $\nbc$ basis and the permutations of the symmetric group which are in one to one correspondence with chambers of the braid arrangement.

\section{Preliminaries}
\label{sec:pre}

Let $\AA$ be an essential affine hyperplane arrangement in $\R^d$, i.e., a set of affine real hyperplanes whose minimal nonempty intersections are points. Let $\FF=\FF(\AA)$ denote the set of closed strata of the induced stratification of $\R^d$. It is customary to endow $\FF$ with a partial ordering $\prec$ given by reverse inclusion of topological closures. The elements of $\FF$ are called {\em faces} of the arrangement. The poset $\FF$ is ranked by the {\em codimension} of the faces. The connected components of $\mathbb{R}^d\setminus \AA$, corresponding to  elements of $\FF$ of maximal dimension, are called {\em chambers}. 
For any $F\in\FF$, denote by $|F|$ the affine subspace spanned by $F$, called the {\em support} of $F$, and set 
$$\AA_F\ :=\ \{ H \in \AA~|~ F\subset H\}.$$

In \cite{salvetti}, Salvetti constructed a regular CW-complex $\SS(\AA)$ (denoted just by $\SS$ if no confusion can arise) that is a deformation retract of the complement
$$\MM(\AA):=\ \C^d \setminus \bigcup_{H\in A}\ H_{\C},$$
 of the complexification of $\AA$.

The  $k$-cells of $\SS$ bijectively correspond to pairs $[C\prec F]$, where $F\in\mathcal{F}$, $\codim(F)=k$ and $C$ is a chamber. A cell $[D\prec G]$ is in the boundary of $[C\prec F]$ if $G \fl F$ and the chambers $D$, $C$ are contained in the same chamber of $\AA_{F}$.

\subsection{Minimal Salvetti's complex}\label{subsec:MC} In \cite{salvettisette}, Salvetti and the first author constructed a minimal complex homotopy equivalent to the complement $\MM(\AA)$ of a complexified real arrangement $\AA$. The main ingredients of this construction are the Forman's Discrete Morse Theory and the Salvetti's complex. They explicitly constructed a combinatorial gradient vector field over $\SS$ whose critical cells correspond to the cells of the minimal complex. This vector field is related to a given system of polar coordinates in $\R^d$ which is \textit{generic} with respect to the arrangement $\AA$. This generic system of coordinates allow them to give a total order $\pol$ on the faces $\FF$ that is the key to describe both, gradient vector field and critical cells. In this paper we are mainly interested in the latter.

More in detail, let $\{V_i\}_{i=0,\ldots,d}$ be a flag of affine subspaces in general position in $\R^d$, such that $\dim(V_i)=i$ for every $i=0,\ldots,d$ and such that the polar coordinates $(\rho,\theta_1,\ldots,\theta_{d-1})$ of every point in a bounded face of $\AA$ satisfy $\rho>0$ and $0< \theta_i < \pi/2$, for every $i=1,\ldots,d-1$ (see \cite[Section 4.2]{salvettisette} for the precise description). Every face $p$ is labeled by the coordinates of the point in its closure that has, lexicographically, least polar coordinates. The {\em polar ordering} associated to a generic flag is the total order $\pol$ on $\FF$ obtained by ordering the faces lexicographically according to their labels. This extends the order in which $V_{d-1}$ intersects the faces in $V_d$ while rotating around $V_{d-2}$. If two faces share the same label,  thus the same minimal point $r$ , the ordering is determined by the general flag induced on the copy of $V_{d-1}$ that is rotated `just past $p$' and the ordering it generates by induction on the dimension (see \cite[Definition 4.7]{salvettisette}).  The $k$-cells of the minimal complex will be the $k$-critical cells (see \cite[Theorem 6]{salvettisette})

$$\textup{Crit}_k(\SS) = \bigg\{[C\prec F] \left\vert\begin{array}{l} \codim(F)=k,\, F\cap V_k \neq \emptyset,\\
G\pol F \textrm{ for all }G \textrm{ with } C\fl G \precneqq F\end{array}\right.\bigg\}$$
(equivalently, $F\cap V_k$ is the maximum in polar ordering among all facets of $C\cap V_k$).

\begin{Notation}We will denote by $\ch(\AA)$ the set of chambers of $\AA$ and by $\textup{Crit}(\SS)=\cup_k\textup{Crit}_k(\SS)$.
\end{Notation}

\begin{Notation}\label{notaz:critfac} Let $p$ be a $k$-face in $\FF$ that intersects $V_k$. The intersection is a point $\overline{p}$ in $V_k$ and we will denote by $V^+_{k-1}(p)$ the copy of $V_{k-1}$ that is rotated in $V_k$ around $V_{k-2}$ ``just past $\overline{p}$" and $V^-_{k-1}(p)$ the copy of $V_{k-1}$ that is rotated in $V_k$ around $V_{k-2}$ ``just before $\overline{p}$" . From now on, when a generic flag $\{V_i\}_{i=0,\ldots,d}$ is given, we will use letter $p$ to denote faces that intersect a $V_k$ for some $k$ and by $\PP(\AA)=\cup_{k=0}^{d} \PP^k(\AA)$ the union of sets
$$
\PP^k(\AA):=\{p \in \FF^k \mid p \cap V_k\neq \emptyset \}
$$
of $k$-codimensional critical faces. Following notations in \cite{salvettisette}, given a chamber $C$ and a facet $p$, we will denote by $C.p$ the unique chamber containing $p$ and lying in the same chamber as $C$ in $\Cal A_{p}$. 
\end{Notation}

Let us remark that, by construction of the polar ordering, all faces $F \prec p$ such that $F \cap V^+_{k-1}(p) \neq \emptyset$ and $F \cap V^-_{k-1}(p) = \emptyset$ verify $F \pog p$.

Moreover, given a chamber $C \in \ch(\AA)$ and a critical face $p \in \PP(\AA)$, the $k$-cell $[C \prec p]$ is critical if and only if $C \cap V^-_{k-1}(p)$ is a bounded chamber in $V^-_{k-1}(p)$. 
In the rest of the paper we will often deal with chambers $C$ such that $C \cap V^-_{k-1}(p)$ (respectively $C \cap V^+_{k-1}(p)$) is bounded in $V^-_{k-1}(p)$ (respectively $V^+_{k-1}(p)$). In this case, for sake of simplicity, we will say that $C$ is bounded in $V^-_{k-1}(p)$ (respectively $V^+_{k-1}(p)$). 
The following remark is straightforward. 

\begin{Remark}\label{rem:rem2} If a chamber $C$ is bounded in $V_k$, then it is bounded in all $V^+_{k-1}(p)$ (respectively $V^-_{k-1}(p)$), $p \in \PP^k(\AA)$, such that $C \cap V^+_{k-1}(p) \neq \emptyset$ (resp. $C \cap V^-_{k-1}(p) \neq \emptyset$). Moreover, if $C$ is bounded in $V^+_{k-1}(p)$ then $C \cap V^-_{k-1}(p) = \emptyset$ and viceversa as $V^+_{k-1}(p)$ and $V^-_{k-1}(p)$ obviously intersect two opposite cones of $\AA_p$. 
\end{Remark}

The argument in Remark \ref{rem:rem2} holds in the more general setting in which $C$ is bounded in the space $V^+_{k}(p^{k+1})$ ( $V^-_{k}(p^{k+1})$), $p^{k+1} \in \PP^{k+1}(\AA)$.

The following Lemma will be useful to prove our main result. 

\begin{Lemma}\label{lem:lem imp} Let $C \prec p^{k+1} \in \PP^{k+1}(\AA)$ be a bounded chamber in $V^-_{k}(p^{k+1})$ and $F=\min_{\pol}\{F^k \in \FF^k \mid F \cap  V^-_k(p^{k+1}) \neq \emptyset, F \prec p^{k+1}\}$. If $p \in \PP^k(\AA)$ is the only $k$-critical face with same support of $F$, then $C.p \pog p$.
\end{Lemma}

In order to prove it we need few more remarks. 

\begin{Remark}\label{rem:rem1}
Let $C \in \ch(\AA)$ be a chamber such that $C \cap V_k \neq \emptyset$ and $p \in \PP^k(\AA)$. If $C$ is bounded in $V_k$, then $C \cap V_i = \emptyset$ for all $i<k$. By definition of polar ordering, $C \pog p$ implies that $C$ is contained in the chamber of $\AA_p$ intersected by $V_{k-1}^+$. Since $C$ and $C.p$ are contained in the same cone of the arrangement $\AA_p$, it follows that $C.p \cap V_{k-1}^+(p) \neq \emptyset$, i.e.  $C.p \pog p$. Furthermore, if $C$ is bounded in $V_k$ then $C.p \cap V_k$ is still bounded in  $V_k$. 
\end{Remark}

Remark \ref{rem:rem1} holds, in particular, if $p \prec p^{k+1}$ for $p^{k+1}$ critical face such that $C \prec p^{k+1}$.
Now let $C \in \ch(\AA)$ be a chamber and $p^{k+1} \in \PP^{k+1}(\AA)$ a critical face such that $C \prec p^{k+1}$. If $C$ is a bounded chamber in $V^-_k(p^{k+1})$ and $F=\min_{\pol}\{F^k \in \FF^k \mid F^k \cap V^-_k(p^{k+1}) \neq \emptyset, F^k \prec p^{k+1}\}$, then $C.F \pog F$. Indeed, by construction, $C$ bounded in $V^-_k(p^{k+1})$ is equivalent to say that $C$ is bounded in $V^-_k(p^{k+1})$ by hyperplanes in $\AA_{p^{k+1}}$ and $C \pol F$ would implies that $C \pol F^k$ for all $k$-faces in her closure, but this is impossible as $C$ is bounded. 

\begin{proof}[Proof of Lemma \ref{lem:lem imp}.] By previous remarks $C.F \pog F$. If $F$ is critical then we are done. Otherwise let $F, F'$ be facets in the set
$$ 
B=\{F^k \in \FF^k \mid F \cap  V^-_k(p^{k+1}) \neq \emptyset, F \prec p^{k+1}\} 
$$ 
and $p,p' \in \PP^k(\AA)$ such that $| p |=| F |$ and $| p' |=| F' |$. Then, by construction of polar ordering, we have that
$$F \pol F' \mbox{ if and only if } p \pol p' .$$
Fix $F=\min_{\pol} B$ then $p=\min_{\pol}\{p^k \in \PP^k(\AA) \mid  p^{k+1} \in | p^k |\}.$ Moreover since $C.p$ and $C$ are contained and bounded in the same cone of the arrangement $\AA_{p^{k+1}}$, it follows that $C.p \cap V_k$ is bounded inside this cone. This implies that if $\overline{p}=\min_{\pol}\{p^k \in \PP^k(\AA) \mid  C.p \prec p^k \}$ then $\overline{p}=p$ and $C.p \pog p$. 
\end{proof}

By minimality of $\SS$, the cardinality of $\textup{Crit}_k(\SS)$ equals the $k$-th Betti number $b_k$ and $\sharp~\textup{Crit}(\SS)=\sum_kb_k=\sharp~\ch(\AA)$. 
In details, if $F \in \FF$ is a face and $C \in \ch(\AA)$ is a chamber, then define $\op_F(C)\in\ch(\AA)$ as the unique chamber such that the set of hyperplanes of $\AA$ that separates $C$ and $\op_F(C)$ equals $\AA_F$. The map
\begin{equation}\label{eq:eta}
\begin{split}
\eta\colon \textup{Crit}(\SS)&\to\ch(A) \\
[C\prec p]&\mapsto\op_p(C)
\end{split}
\end{equation}
is a bijection (see \cite{salvettisette}). 

Remark that a chamber $C$ is bounded in $V^-_{k-1}(p)$ (i.e. $[C\prec p]$ is critical) if and only if $\op_{p}(C)$ is bounded in $V^+_{k-1}(p)$ and the following Lemma holds.

\begin{Lemma}\label{lem:lemfond} If $[C \prec p]$ is a $k$-critical cell and $\breve{p}=\min_{\pol}\{p^{k-1} \in \PP^{k-1}(\AA) \mid p \in | p^{k-1} | \}$, then $[\op_{\breve{p}}(C.\breve{p}) \prec \breve{p}]$ is a $(k-1)$-critical cell.
\end{Lemma}

\begin{proof} If $[C \prec p]$ is a $k$-critical cell, then $C$ is a bounded chamber in $V^-_{k-1}(p)$ and, by Lemma \ref{lem:lem imp}, $C.\breve{p} \pog \breve{p}$. It follows that $C.\breve{p} \cap V^+_{k-2}(\breve{p}) \neq \emptyset$ is a bounded chamber in $V^+_{k-2}(\breve{p})$ and hence $\op_{\breve{p}}(C.\breve{p})$ is a bounded chamber in $V^-_{k-2}(\breve{p})$, that is $[\op_{\breve{p}}(C.\breve{p}) \prec \breve{p}]$ is a $(k-1)$-critical cell.
\end{proof}

In \cite{Delucchi2010combinat}, Delucchi and the first author gave a more general combinatorial description of the minimal complex constructed in \cite{salvettisette} and briefly described in this Section. They proved that it is possible to obtain the same result of \cite{salvettisette} replicing the flag $\{V_i\}_{i=0,\ldots,d}$ with a flag of pseudospaces and \textit{flipping} the pseudohyperplane $\Cal{V}^-_{k-1}(p)$ around $p \in \PP^k(\AA)$ inside the pseudospace $\Cal{V}_{k}$ instead of rotating the space $V^-_{k-1}(p)$ past $p$ inside the space $V_k$ around $V_{k-2}$.
They called \textit{special ordering} an order obtained using a flag of pseudospaces that allows to describe the minimal complex analogously to the construction in  \cite{salvettisette}.             
This description turn out to be very useful for a better understanding of the combinatorics underling the construction in \cite{salvettisette} and to describe a special class of arrangements, called recursively orderable, that admits a special ordering with very handy description. They also proved that supersolvable arrangements are recursively orderable.

\subsection{Minimal complex for supersolvable arrangements} The class of ``strictly linearly fibered'' arrangements was introduced by Falk and Randell \cite{FaRa} in order to generalize the techniques of Fadell and Neuwirth's proof \cite{FaNe} of asphericity of the braid arrangement (involving a chain of fibrations). Later on, Terao \cite{terao1} recognized that strictly linearly fibered arrangements are exactly those which intersection lattice is { supersolvable} \cite{stanley1}. Since then these arrangements are known as {\em supersolvable arrangements}, and deserved intense consideration. See \cite{orlterao}, for more details.

\begin{Definition}\label{supers} A central arrangement $\AA$ of complex hyperplanes in $\mathbb{C}^d$ is called \emph{supersolvable}
if there is a filtration
$\AA=\AA_d \supset \AA_{d-1} \supset \cdots \supset \AA_2 \supset \AA_1$
such that
\begin{enumerate}
\item[(1)] $\rank(\AA_i)=i$ for all $i=1,\ldots ,d$ ;
\item[(2)] for every two $H, H^{\prime} \in \AA_i$ there exits
some $H^{\prime \prime} \in \AA_{i-1}$ such that 
$H \cap H^{\prime} \subset H^{\prime \prime}$.
\end{enumerate}
\end{Definition}

Let $\AA$ denote an affine real arrangement of hyperplanes in $\mathbb{R}^d$. A flag $\{V_k\}_{k=0,\ldots,d}$ of affine subspaces is called a {\em general flag} if every one of its subspaces is in general position with respect to $\AA$ and if, for every $k=0,\ldots d-1$, $V_k$ does not intersect any bounded chamber of the arrangement $\AA\cap V_{k+1}$. Note that this is a less restrictive hypothesis than the one required for being a {\em generic} flag in \cite{salvettisette}. 

\begin{Remark}\label{generalext} Let $\AA$ be as in Definition \ref{supers} and consider the arrangement $\AA_{d-1}$ in $\mathbb{R}^d$. It is clearly not essential, and the top element of the intersection poset $\LL(\AA_{d-1})$ is a $1$-dimensional line that we may suppose to coincide with the $x_1$-axis. The arrangement $\AA_{d-1}$ determines an essential arrangement on any hyperplane $H$ that meets the $x_1$-axis at some $x_1=t$. 
For all $t$, the intersection of $\AA_{d-1}$ with the hyperplane $H$ determines an essential, supersolvable arrangement $\AA'_{d-1} \subset \mathbb{R}^d$ with $\AA'_{r}=\AA_r$ as sets, for all $r\leq d-1$. Thus, given a flag of general position subspaces for $\AA'_{d-1}$, we can find a combinatorially equivalent flag $\{V_k\}_{k=0,\ldots, d-2}$ on  $H$.

Now let us consider a hyperplane $H$ in $\mathbb{R}^d$ that is orthogonal to the $x_1$-axis, and suppose we are given on it, as above, a valid flag $\{V_k\}_{k=0,\ldots,d-2}$ of general position subspaces for $\AA_{d-1}$. By tilting $H$ around $V_{d-2}$ 
we can  obtain a hyperplane $H'$ that is in general position with respect to $\AA$ and for which all points of $\AA\cap H'$ are on the same side with respect to $V_{d-2}$, and for which $V_0$ lies in an unbounded chamber.

By setting $V_{d-1}:=H'$ and $V_d:=\mathbb{R}^d$ we thus obtain a valid general flag for $\AA=\AA_d$. Define $\PP^k(\AA_d)$ as the points of $\AA_d\cap V_k$ and analogously for $\PP^k(\AA_{d-1})$. The flag remains general by translating $H'=V_{d-1}$ in $x_1$-direction away from the origin: we can therefore suppose that there is $R\in\mathbb{R}$ such that  for all $k$, $k=1,\ldots,d-1$, every element of $\PP^k(\AA_{d-1})$ is contained in a ball of radius $R$ centered in $V_0$, that contains no element of $\PP^k(\AA_{d})\setminus\PP^{k}(\AA_{d-1})$.
\end{Remark}

\begin{Definition}[Recursive Ordering]\label{df:maindef} Let $\AA$ be a real arrangement and $\{V_k\}_{k=0,\ldots, d}$ a general flag. The corresponding {\em recursive ordering} is the total ordering $\totl$ of $\PP(\AA)$ given by setting $p\totl r$ if one of the following occurs
\begin{enumerate}
\item[(i)] $p\in\PP^h(\AA)$, $r\in\PP^k(\AA)$ for $h<k$;
\item[(ii)] there is $k$ such that $p,r\in\PP^k(\AA)$ and we can write $p_0:=\min_{\totl} \{p'\in\PP^{k-1}(\AA)\mid p\subset \vert p'\vert\}$,  $r_0:=\min_{\totl} \{p'\in\PP^{k-1}(\AA)\mid r\subset \vert p'\vert\}$, \begin{enumerate}
\item[(a)] either $p_0 \totl r_0$,
\item[(b)] or $p_0=r_0$ and there exists a sequence of faces
$$p_0 \fl p_1 \fg r_1 \fl p_2 \fg r_2 \cdots \fl p,$$
such that $\codim(p_i)=\codim(r_i)+1=\codim(p)$, and every $r_i$, $p_i$ intersect $|p_0| \cap V_k$, and $p_i\neq r$ for  all $i$.
\end{enumerate}\end{enumerate}
\end{Definition}

\begin{Theorem}\label{ssfol} Any supersolvable complexified arrangement $\AA$ is recursively orderable. 
Moreover, the recursively orderable special ordering $\pol$ can be chosen so that
for all $i=2,\ldots,d$ and all $k=1,\ldots, i-1$, if  $p_1\in\PP^k(\AA_{i-1})$ and  $p_2\in\PP^k(\AA_{i}) \setminus \PP^k(\AA_{i-1})$  lie in the support of the same $(k+1)$-codimensional face,
then $p_1 \pol p_2$.
\end{Theorem}

\begin{Corollary}\label{segmentato}
Let $\AA$ and $\{V_k\}_{k=1,\ldots,d}$ be as in the construction of Remark \ref{generalext}. Then, for every $k=1,\ldots,d$, if $p_1\in\PP^k(\AA_{d-1})$ and $p_2\in\PP^k(\AA)\setminus \PP^{k}(\AA_{d-1})$ are both contained in the support of the same $p\in\PP^{k-1}(\AA)$, then $p_1\pol p_2$ in every special ordering of $\PP^k(\AA)$.
\end{Corollary}

Any order $\pol$ on the faces $\FF$  induced by a general flag $\{V_k\}_{k=0,\ldots,d}$ induces an order $\pol_{\AA}$ on hyperplanes of $\AA$ as follows 
$$H \pol_{\AA} H^{\prime} \mbox{ if and only if } p_H \pol p_{H^{\prime}},$$
$p_H, p_{H^{\prime}} \in \PP^1(\AA)$ being the only two faces such that  $|p_H|=H, |p_{H^{\prime}}|= H^{\prime}$.

By the Theorem \ref{ssfol} and the Corollary \ref{segmentato}, the order $\pol$ can be chosen in such a way that the following property holds
\begin{equation}\label{eq:propord}
\mbox{if } H \in \AA_{i} \setminus \AA_{i-1}, H^{\prime} \in \AA_{j} \setminus \AA_{j-1} \mbox{ with } i <j, \mbox{ then } H \pol_{\AA} H^{\prime}. 
\end{equation}
As no confusion can arise, we will denote the order $\pol_{\AA}$ simply by $\pol$.

\subsection{\textit{nbc}-basis for supersolvable arrangements.}\label{sec:nbcdef}

Let us briefly recall some basic fact on the Orlik-Solomon algebra and its $nbc$-basis.

Fix an arbitrary order $\pol$ on a central arrangement $\AA$ in $\R^d$. Then an ordered $k$-uple $(H_1,\ldots,H_{k+1})$, with $H_1 \pol \ldots \pol H_{k+1}$, is \emph{independent} if $\rank(\cap_{i=1}^{k+1} H_i)=k+1$, and it is \emph{dependent} otherwise. It is called a \textit{circuit} if it is minimally dependent, that is 
$(H_1,\ldots,H_{k+1})$ is dependent, while $(H_1,\ldots,\hat{H}_p, \ldots, H_{k+1})$ is independent for any $1 \leq p \leq k$. 

An ordered independent $k$-uple $(H_1,\ldots,H_{k})$ is a \emph{broken circuit} if it exists an hyperplane $H \pol H_1$ such that $(H,H_1,\ldots,H_{k})$ is a circuit.
It is well known that a basis for the Orlik-Solomon algebra of the arrangement $\AA$ is given by all ordered independent $k$-uples $(H_1,\ldots,H_{k})$, $0 \leq k \leq d$, that do not contain any broken circuit. Such a basis is called a \textit{non broken circuit basis}, or simply \textit{nbc-basis}.

Let $\AA \subset \R^d$ be a central supersolvable arrangement.  Bj{\"o}rner and Ziegler (see \cite{bjorner1991broken}) proved that in a supersolvable arrangement a  $k$-uple $(H_1,\ldots,H_{k})$ does not contain a broken circuit if and only if it does not contain a $2$-broken circuit. From which we get the following Proposition.

\begin{Proposition} Let $\AA$ be a supersolvable arrangement in $\R^d$ together with an order $\pol$ that verifies property \eqref{eq:propord}. Then the set \begin{equation*}
\nbc_k(\AA):=\{(H_{1}, \dots, H_{k}) \in \AA^k \mid H_{j}\in\AA_{i_j}\setminus\AA_{i_j-1}, i_j<i_{j+1}\} 
\end{equation*} 
is a \textit{nbc}-basis of the $k$-stratum of the Orlik-Solomon algebra associated to $\AA$.
\end{Proposition}

\begin{proof}\footnote{This Proposition is a quite known fact but we could not find a detailed proof of it anywhere, so we provided it here.} By Bj{\"o}rner and Ziegler's  result, it is enough to check that $k$-uples $(H_{1}, \dots, H_{k}) \in \nbc_k$ do not contain couples $(H_i,H_j)$ that are broken circuits. Since $\AA$ is a supersolvable arrangement, if $H_i,H_j$ are hyperplanes that belong to the same  subarrangement $\AA_{h_i+1} \setminus \AA_{h_i}$  then it exists $H \in \AA_{h_i}$ such that $H_i \cap H_j \subset H$, that is $(H,H_i,H_j)$ is a broken circuit. 

On the other hand, if $H_i \in \AA_{h_i+1} \setminus \AA_{h_i}$ and $H_j \in \AA_{h_j+1} \setminus \AA_{h_j}$ belong to different subarrangements with $h_i<h_j$, then for any $H \pol H_i$ we get that $\rank(H \cap H_i \cap H_j)=3$. Indeed if $H \in \AA_h$, $h<h_i$ this is obvious while, if $H \in \AA_{h_i}$ then it exists $H^{\prime} \in \AA_{h_i-1}$ such that $H \cap H_i=H^{\prime} \cap H_i$ and $\rank(H \cap H_i \cap H_j)=\rank(H^{\prime} \cap H_i \cap H_j)=3$.
 \end{proof}
 Following the previous Proposition we will denote $$\nbc(\AA):=\cup_k\nbc_k (\AA). $$
 For the seek of simplicity, when no confusion arises, we will omit $\AA$ in the rest of the paper and we will simply denote $\nbc_k(\AA)$ by $\nbc_k$ and $\nbc(\AA)$ by $\nbc$. Similarly, we will simply denote $\PP^k(\AA)$ by $\PP^k$ and $\PP(\AA)$ by $\PP$.

\section{Orlik-Solomon algebra and minimal complex}\label{sec:main}

In this Section $\AA$ is a supersolable arrangement in $\R^d$ endowed with a recursive special ordering $\pol$ induced by a generic flag $\{V_k\}_{k=0,\ldots,d}$ of affine subspaces as in Remark \ref{generalext}, i.e. $\pol$ satisfies the conditions in Theorem \ref{ssfol}.

\subsection{A natural relation}

If  $|p|=\cap_{j=1}^mH'_{j}$ is the support of $p$, then $p$ is the only $k$-codimensional face that contains the intersection $|p| \cap V_k$. That is there is a natural bijection between elements of the intersection poset $\LL(\AA)$ and critical faces. Moreover, by the properties of Definition \ref{supers}, to get a $k$-codimensional intersection $|p|$ in the poset $\LL(\AA)$ of a supersolvable arrangement it is enough to consider a $k$-uple $(H_{1}, \dots, H_{k}) \in \nbc_k$ such that $\cap_{i=1}^kH_{i}=\cap_{j=1}^mH'_{j}=|p|$. 

Notice that the latter is not a bijection. With the previous notations, if $H_j \in \AA_{i_j}\setminus\AA_{i_j-1}$ and $H \neq H_j$ is another hyperplane in  $\AA_{i_j}\setminus\AA_{i_j-1}$ that contains $p$,  then $(H_{1}, \dots, H_{j}, \dots, H_{k}) $ and $(H_{1}, \dots, H, \dots, H_{k})$, with $H$ in the $j$-th position, are both $k$-uples in $\nbc_k$ with intersection equals the support $|p|$ of $p$.

Let $H \in \AA_{i_j}\setminus\AA_{i_j-1}$ be an hyperplane that contains the critical face $p$, we define the set
$$[H]_p:=\{H'\in\AA_{i_j}\setminus\AA_{i_j-1}~|~p\subset H'\}.$$ 
  
Then to any critical $k$-codimensional face $p$ is attached one and only one $k$-uple of classes of hyperplanes 
\begin{equation}\label{eq:classi}
[p]:=([H_{1}]_p, \dots, [H_{k}]_p).
\end{equation}

It is an easy remark that, if $p^{\prime} \prec p$ is a $(k-1)$-critical face, then it exists an index $1 \leq j \leq k$ such that $[p^{\prime}]=([H^{\prime}_{1}]_{p^{\prime}}, \ldots, \widehat{[H^{\prime}_{j}]}_{p^{\prime}}, \ldots, [H^{\prime}_{k}]_{p^{\prime}})$, $H^{\prime}_i \in [H_i]_p$.
Remark that the inclusion $[H^{\prime}_{1}]_{p^{\prime}} \subseteq [H^{\prime}_{1}]_{p}$ holds.
\begin{Definition} Given two chambers $C, C'\in\ch(\AA)$ and an hyperplane $H$ in $\AA$, we define
\begin{equation*}
    (C~|~C')_H :=
    \begin{cases}
      -1 & \text{if } H \text{ separates } C \text{ and } C',\\
      1 & \text{otherwise }.
    \end{cases}
\end{equation*}
\end{Definition}
With previous notations, define
\begin{equation}\label{df:nablak}
\begin{split}
f_k\colon & \textup{Crit}_k(\SS)\to\nbc_k
\end{split}
\end{equation}

as $f_k([C\prec p])=(H_1, \dots, H_k)$ if and only if
\begin{enumerate}
\item[(i)]$\cap_{i=1}^k H_i=|p|$;
\item[(ii)] if $(-1)^{k-j}=-1$, then $$H_j= \min_{\pol}\{ H \in [H_j]_p ~|~ (C~|~C_0)_H=-1 \};$$
\item[(iii)] if $(-1)^{k-j}=1$,  then 
$$H_j= \max_{\pol}\{ H \in [H_j]_p ~|~ (C~|~C_0)_H=1 \};$$
\item[(iv)] $H_k$ is a wall of $C$ and $(C~|~C_0)_{H_k}=1$.
\end{enumerate}

The above relation naturally define a relation 
\begin{equation}\label{df:nabla}
f\colon \textup{Crit}(\SS)\to\nbc
\end{equation}
between the critical cells of Salvetti's complex and the $nbc$-basis of Orlik-Solomon algebra of the supersolvable arrangement $\AA$.
 
\subsection{Bijection between $\nbc$ and critical cells.} In this Section, we prove that $f_k$, $k=0,\ldots,d$, are well defined bijective maps.

\begin{Lemma}\label{lem:sottofaccia} If $p \in \PP^k$ is a $k$-critical face with $[p]=([H_1]_p, \dots, [H_k]_p)$, then it exists one and only one $(k-1)$-critical face $\breve{p}$ such that $\breve{p}\prec p$ and $[\breve{p}]=([H_1]_{p}, \dots, [H_{k-1}]_{p})$.
\end{Lemma}

\begin{proof}Let $p \in \PP^k$ be a $k$-critical face with $[p]=([H_1]_p, \dots, [H_k]_p)$, $H_i \in \AA_{h_i} \setminus \AA_{h_{i-1}}$ and let $p_1,p_2 \prec p$ be two $(k-1)$-critical faces with $[p_1]=([H'_1]_{p_1},\ldots,[H'_{k-1}]_{p_1})$ and $[p_2]=([H''_1]_{p_2},\ldots,[H''_{k-1}]_{p_2})$, where $H'_{i}, H''_{i} \in [H_i]_{p}$, $i=1,\ldots,k-1$. If $p_1 \neq p_2$, then $\mid p_1 \mid \cap \mid p_2 \mid$ would be a space of codimension $ \geq k$ that contains $p$ and this is not possible as $\mid p_1 \mid \cap \mid p_2 \mid$ is an element in the intersection lattice of the arrangement $\AA_{h_{k-1}}$ while $ \mid p \mid \subset H_k$ and $H_k \in \AA_{h_k} \setminus \AA_{h_{k-1}}$   . Then there is a unique $(k-1)$-critical face $\breve{p} \prec p$, $[\breve{p}]=([\breve{H}_{1}]_{\breve{p}}, \ldots , [\breve{H}_{k-1}]_{\breve{p}})$ and it follows that $[\breve{H}_{i}]_{\breve{p}}=[H_{i}]_{p}$ for any $i=1,\ldots,k-1$.
\end{proof}

Fix a $k$-critical cell $[C \prec p]$ and $\breve{p}$ as in Lemma \ref{lem:sottofaccia}. 

\begin{Lemma}\label{lem:opp} If $[C \prec p]\in \textup{Crit}_k(\SS)$ is a $k$-critical cell then $[\op_{\breve{p}}(C.\breve{p}) \prec \breve{p}]$ is a $(k-1)$-critical cell.
\end{Lemma}

\begin{proof} It follows from Lemma \ref{lem:lemfond} and the fact that $\breve{p}=\min_{\pol}\{p' \in \PP^k \mid p' \prec p\}$, as $\pol$ is a recursive order.
\end{proof}

Let $\widetilde{C}.\breve{p}$ be the chamber of the arrangement $\AA_{\breve{p}}$ that contains the chamber $C.\breve{p}$ and hence $C$. Let $[C'\prec p]$ be another $k$-critical cell with $C' \neq C$ and $C' \subset \widetilde{C}.\breve{p}$. Then $C$ and $C'$ have to be separated by at least one hyperplane and, as they belong to the same chamber of $\AA_{\breve{p}}$, they are separated by an hyperplane $H \in [H_k]_{p}$. It is also an easy remark that $(C\mid C_0)_H=1$ if and only if $(C'\mid C_0)_H=-1$.

Viceversa, any hyperplane $H \in [H_k]_{p}$ intersects the chamber $\widetilde{C}.\breve{p}$ and hence it is the separating hyperplane of two different chambers contained in $\widetilde{C}.\breve{p}$. That is for each hyperplane $H \in [H_k]_{p}$ there is a chamber $C' \subset \widetilde{C}.\breve{p}$ such that $H$ is a wall of $C'$ and  $(C'\mid C_0)_H=1$ and we proved the following Lemma.

\begin{Lemma}\label{lem:muro} Let $p \in \PP^k$ be a $k$-critical face, $[p]=([H_1]_p, \dots, [H_k]_p)$. If $[C \prec p] \in \textup{Crit}_k(\SS)$ is a $k$-critical cell, then $C$ has a wall $H \in [H_k]_p$ that satisfies $(C\mid C_0)_H=1$. 
\end{Lemma}

\begin{Lemma}\label{lem:unicmuro} The wall in Lemma \ref{lem:muro} is unique.
\end{Lemma}
\begin{proof} Let $[C \prec p] \in \textup{Crit}_k(\SS)$ be a $k$-critical cell, $[p]=([H_1]_p, \dots, [H_k]_p)$, and $H,H' \in [H_k]_p$ two hyperplanes satisfying Lemma \ref{lem:muro}. Then, by supersolvability, it exists $\overline{H} \in [H_{k-1}]_p$ such that $H \cap H' \subset \overline{H}$. It follows that $\rank(H \cap H' \cap  \overline{H})=2$ and hence $H,H'$ and $\overline{H}$ cannot be walls of the same chamber $C$.  
\end{proof}

\begin{Theorem}\label{th:map} The relation $f_k$ defined in \eqref{df:nablak} describes a bijection between $\textup{Crit}_k(\SS)$ and $\nbc_k$. 
\end{Theorem}
\begin{proof} We will prove the theorem by induction on the dimension $k$ of the critical cells in $\textup{Crit}(\SS)$. The theorem holds trivially for the $0$-critical cell that corresponds to the empty set.  

Let $[C \prec p]$ be a $k$-critical cell. Then, by Lemma \ref{lem:opp}, the $(k-1)$-cell $[\op_{\breve{p}}(C.\breve{p})\prec \breve{p}]$ is critical and, by inductive hypothesis, it exists one and only one $(k-1)$-uple of hyperplanes  $(H_{1},\ldots,H_{k-1}) \in \nbc_{k-1}$ such that  
$(H_{1},\ldots,H_{k-1})=f_{k-1}([\op_{\breve{p}}(C.\breve{p})\prec \breve{p}])$. 
Moreover, since $C$ and $C.\breve{p}$ belong to the same chamber $\widetilde{C}.\breve{p}$ of the arrangement $\AA_{\breve{p}}$ and $C.\breve{p}$ and $\op_{\breve{p}}(C.\breve{p})$ are opposite chambers with respect to $\breve{p}$, it follows that 
$$\min_{\pol}\{ H \in [H_j]_p ~|~ (C~|~C_0)_H=-1 \}=$$
$$\min_{\pol}\{ H \in [H_j]_p ~|~ (C.\breve{p}~|~C_0)_H=-1 \}=$$
$$\max_{\pol}\{ H \in [H_j]_p ~|~ (\op_{\breve{p}}(C.\breve{p})~|~C_0)_H=1 \},$$
 for $j=1,\dots, k-1$ and, analogously,
$$\max_{\pol}\{ H \in [H_j]_p ~|~ (C~|~C_0)_H=1 \}=$$
$$\max_{\pol}\{ H \in [H_j]_p ~|~ (C.\breve{p}~|~C_0)_H=1 \}=$$
$$\min_{\pol}\{ H \in [H_j]_p ~|~ (\op_{\breve{p}}(C.\breve{p})~|~C_0)_H=-1 \},$$
for  $j=1,\ldots, k-1$.

Then, if $H \in [H_{k}]_{p}$ is the only hyperplane satisfying Lemma \ref{lem:muro}, the $k$-uple $(H_{1},\ldots,H_{k-1}, H)$ satisfies all conditions of \eqref{df:nablak} and it is clearly the only element in $f_k([C \prec p])$, that is $f_k$ is a map. 

We need to verify that $f_k$ is bijective. Since $\nbc_k$ and $\textup{Crit}_k(\SS)$ are finite sets of the same cardinality, it is enough to show that $f_k$ is injective. Let  $[C' \prec p]$ be a $k$-critical cell with $f_k([C' \prec p])=f_k([C \prec p])$. Then, by inductive hypothesis, $\op_{\breve{p}}(C.\breve{p})=\op_{\breve{p}}(C'.\breve{p})$, that is $C.\breve{p}=C'.\breve{p}$ and hence $\widetilde{C}.\breve{p} = \widetilde{C}'.\breve{p}$. By Lemmas \ref{lem:muro} and \ref{lem:unicmuro} there is only one hyperplane $H' \in [H_k]_{p}$ satisfying condition (iv) in the definition of the map $f_k$, that is $C=C'$ and $f_k$ is injective.
\end{proof}
As immediate corollaries we get the following results.

\begin{Corollary}\label{correspcelluple} The map $f$ defined in Section \ref{sec:main} is a bijection.
\end{Corollary}

\begin{Corollary} The map $\eta^{-1}f$ is a bijection between $\ch(\AA)$ and $\nbc$.
\end{Corollary}

\begin{Remark} In the construction of the map $f_k$, we used properties of supersolvable arrangements, but the main argument behind its description and construction is merely a geometrical one. In fact, the map could be given in general for any complexified arrangement analogously to what has been done by other authors, such as Yoshinaga in \cite{yoshinaga2009chamber}, Delucchi in \cite{Del2008} and Gioan and Las Vergnas in \cite{GioLa}. The interest of this map is its handy and natural description that allows applications such as the one in the subsequent Sections. A natural question is to which extent and how this map can be generalized without losing its simple description. A partial answer about the non triviality of this question is given by the following counterexample.
\end{Remark}

\subsection{Nice arrangements} 
A natural generalization of the notion of supersolvable arrangements is the one of  nice arrangements introduced by Terao in \cite{terao1992factorizations}. See also \cite{orlterao}.

Fix an arrangement $\AA$ in $\mathbb{R}^d$. A partition $\pi=(\pi_1,\dots,\pi_s)$ of $\AA$ is called \emph{independent} if for any $H_i\in\pi_i\subset\AA$,  the hyperplanes $H_1,\dots,H_s$ are independent, i.e. $\rank(H_1\cap\cdots\cap H_s)=s$.

Consider now $X\in L(\AA)$ and $\pi=(\pi_1,\dots,\pi_s)$ a partition of $\AA$. Then the \emph{induced partition} $\pi_X$ is a partition of the arrangement $\AA_X$ whose blocks are the subsets $\pi_i\cap\AA_X$, for $i=1,\dots,s$, which are not empty. 

\begin{Definition}\label{nice} A partition $\pi=(\pi_1,\dots,\pi_s)$ of $\AA$ is called \emph{nice} if
\begin{enumerate}
\item $\pi$ is independent;
\item for any $X\in L(\AA)$, the induced partition $\pi_X$ contains a block which is a singleton unless $\AA_X=\emptyset$.
\end{enumerate}
We will call $\AA$ a \emph{nice arrangement} if it admits a nice partition.
\end{Definition}
A supersolvable arrangement $\AA$ is a nice arrangement with $s=d-1$ and $\pi_i=\AA_{i+1}\setminus \AA_i$.

Nice arrangements have been introduced by Terao since they answered the question of which arrangements have their Orlik-Solomon algebra factorizable. In particular,  $\pi=(\pi_1,\dots,\pi_s)$ is a nice partition of $\AA$ if and only if the Orlik-Solomon algebra of $\AA$, viewed as $\mathbb{Z}$-module factorizes as
$$A(\AA)=(\mathbb{Z}\oplus B(\pi_1))\otimes\cdots\otimes (\mathbb{Z}\oplus B(\pi_s)),$$
where $B(\pi_i)$ denotes the submodule of $A^1(\AA)$ spanned by the hyperplanes in $\pi_i$.

\begin{figure}[htbp]
\begin{picture}(90,90)(0,0)
\thicklines

\put(10,0){\line(1,1){90}}
\put(-5,65){\line(1,0){100}}
\put(45,0){\line(0,1){90}}
\put(80,0){\line(-1,1){90}}

{\small 
\put(4,10){$H_2$}
\put(70,10){$H'_2$}
\put(0,55){$H_1$}
\put(46,0){$H_3$}
\put(80,35){$C_0$}
\put(46,50){$C_1$}
\put(31,50){$C_2$}
}
\end{picture}
      \caption{Nice arrangement.}
\label{nicearr}
\end{figure}
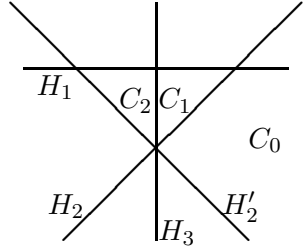

\begin{Example} Let us consider the arrangement $\AA$ described in Figure \ref{nicearr}. The cone $c\AA$ over $\AA$ is supersolvable and  hence nice. Consider the partition defined by $\pi_1:=\{H_1\}, \pi_2:=\{H_2,H'_2\}$ and $\pi_3:=\{H_3\}$. This partition is nice, but $\pi_1 \subset \pi_1 \cup \pi_2$ is not supersolvable arrangement, that is the partition $\pi=(\pi_1,\pi_2,\pi_3)$ is not compatible with supersolvable structure of $c\AA$.
If we replace the $\nbc$ basis obtained by supersolvable filtration $\AA=\{H_1,H_2,H_2',H_3\} \supset \AA_2=\{H_1,H_2\} \supset \{H_1\}$ with the one obtained using partition $\pi$, the map $f_2$ defined in \ref{df:nablak} should associate to the $2$-critical cell $[C_2 \prec p]$ a $3$-uple of hyperplanes with $H_3$ as last entry since $\pi_3:=\{H_3\}$. But this clearly does not satisfy condition (iv) in  the definition of the map $f_k$.  
\end{Example}

Let us remark that in the above Example no recursive order is compatible with the nice partition $\pi$.

\section{On the (co)homology group of the Milnor fiber}\label{sec:mil}

In this Section we present an application of the map constructed in Section \ref{sec:main} to computations on the first (co)homology group of the Milnor fiber of supersolvable arrangements.

\subsection{Boundary map} In \cite{salvettisette}, the authors described a boundary operator $\partial_*$ such that the minimal complex  $(\mathcal{C}(\SS)_*,\partial_*)$, built from $\textup{Crit}_*(\SS)$, computes the homology groups with local coefficients $H_k(M(\mathcal A),\mathbb{L})$. In particular, in degree $k$ $$\mathcal{C}(\SS)_k=\mathcal{C}_k:=\mathbb{L}.e_{[C\prec p^k]},$$ where one has one generator for each $[C\prec p^k]\in\textup{Crit}_k(\SS)$.

A $(k-1)$-critical face $[D\prec p^{k-1}]$ is in the boundary of a $k$-critical face $[C\prec p^k]$ if and only if it exists an \textit{ordered admissible sequence} for the pair $[C\prec p^k], [D\prec p^{k-1}]$. More in detail, given $p^k$ a critical facet of codimension $k$, a sequence of pairwise different facets of codimension $k-1$
$$\mathcal F(p^k) \ :=\ (F^{(k-1)}_{i_1},\cdots,F^{(k-1)}_{i_m}), \ m\geq 1$$
such that
$$F^{(k-1)}_{i_j}\prec p^{k},\ \forall \ j$$
and
$$p^k\pol F^{(k-1)}_{i_j}\ \text{for}\ j<m$$
while for the last element
$$F^{(k-1)}_{i_m}\pol p^k$$
is called an  \emph{admissible $k$-sequence}.\\
Furthermore, it is called an \emph{ordered} admissible $k$-sequence if
$$F^{(k-1)}_{i_1}\pol\cdots \pol F^{(k-1)}_{i_{m-1}}.$$
Two admissible $k$-sequences
$$\mathcal F(p^k) \ :=\ (F^{(k-1)}_{i_1},\cdots,F^{(k-1)}_{i_m})$$
$$\mathcal F(p'^k) \ :=\ (F'^{(k-1)}_{j_1},\cdots,F'^{(k-1)}_{j_l})$$
$p^k\neq p'^k,$ can be {\it composed} into a sequence
$$\mathcal F(p^k)\mathcal F(p'^k) \ :=\
(F^{(k-1)}_{i_1},\cdots,F^{(k-1)}_{i_m},
F'^{(k-1)}_{j_1},\cdots,F'^{(k-1)}_{j_l})$$ if
$$F^{(k-1)}_{i_m}\prec p'^k.$$

Given a critical $k$-cell  $[C\prec p^k]$ and a critical $(k-1)$-cell  $[D\prec p^{k-1}]$ \ an \textit{ordered admissible sequence} for the given pair of critical cells is a sequence of facets of codimension $k-1$
$$\ \mathcal F_{([C\prec p^k],\ [D\prec p^{k-1}])} :=\ (F^{(k-1)}_{i_1},\cdots,F^{(k-1)}_{i_h})$$
obtained as composition of ordered admissible $k$-sequences
$$\mathcal F(p^k_{j_1})\cdots \mathcal F(p^k_{j_s})$$
such that
\medskip

a) \ $p^k_{j_1}=p^k$ \ (so $F^{(k-1)}_{i_1}\prec\ p^{k}$);

b) \ $F^{(k-1)}_{i_h}\ =\ p^{k-1}$ \ and the chamber
$$C.F^{(k-1)}_{i_1}.\cdots.F^{(k-1)}_{i_h}$$
(see Notation \ref{notaz:critfac}) equals \ $D$;

c) \ for all $j=1,\dots, h$ \ the $(k-1)$-cell \
$$[C.F^{(k-1)}_{i_1}.\cdots.F^{(k-1)}_{i_j}\prec F^{(k-1)}_{i_j}]$$
is locally critical, that is $$F^{(k-1)}_{i_j}=\max_{\pol}\{F'~|~C.F^{(k-1)}_{i_1}.\cdots.F^{(k-1)}_{i_j}\prec F'\prec F^{(k-1)}_{i_j}\}.$$

Denote by
$${\mathcal Seq}\ =\ {\mathcal Seq}([C\prec p^k],[D\prec p^{(k-1)}])$$
the set of all admissible sequences for the given pair of critical
cells. 

Of course, this is a finite set which is determined only by the orderings $\prec$ and  $\pol.$ In fact, the ``operation'' which
associates to a chamber $C$ and a facet $p$ the chamber $C.p$ is detected only by the Hasse diagram of the partial ordering $\prec.$

In addition, if $\AA$ is supersolvable, Theorem \ref{ssfol} holds and $\AA$ is recursively orderable.
The order $\pol$ can be chosen so that
for all $i=2,\ldots,d$ and all $k=1,\ldots, i-1$, if  $p_1\in\PP^k(\AA_{i-1})$ and  $p_2\in\PP^k(\AA_{i}) \setminus \PP^k(\AA_{i-1})$  lie in the support of the same $(k+1)$-codimensional face,
then $p_1 \pol p_2$.

From these considerations it follows that, if a $k$-critical cell $[C\prec p^k]$ is such that $p_k\in\PP^k(\AA_{i-1})$, then all $(k-1)$-facets $F^{k-1}_{i_j}$ that belong to ordered admissible sequences originating from $[C\prec p^k]$ are in $\FF^{k-1}(\AA_{i-1})$. Indeed all facets in $\FF(p^k)$ clearly belong to $\FF^{k-1}(\AA_{i-1})$. Moreover, from the conditions
$F^{(k-1)}_{i_m}\pol p^k$, it follows that if $\FF(p'^k)$ is a sequence composed with $\FF(p^k)$, then by definition, $p'^k \pol F^{(k-1)}_{i_m}\pol p^k$ and, from the special choice of the recursive order $\pol$, we get that $p'^k \in \PP^k(\AA_{i-1})$. 

Moreover the coefficients in the local system $\mathbb{L}$ depend only on the monodromy around hyperplanes that separate the chambers $C$ and $D$ and the following statement is proved. 

\begin{Theorem}\label{thm:inclus} If $\AA$ is a supersolvable arrangement, it exists a recursive order $\pol$ such that the inclusion map
$$
i_{h,k}\colon (\mathcal{C}(\SS(\AA_h))_*,\partial_*) \longrightarrow (\mathcal{C}(\SS(\AA_k))_*,\partial_*)
$$
is a well defined inclusion of algebraic complexes, for $h<k$.
\end{Theorem}

Following construction in \cite{settepanella2009cohomology} (see also \cite{DeProSal}), the inclusion map 
$$
i_{j-1,j}\colon (\mathcal{C}(\SS(\AA_{j-1}))_*,\partial_*) \longrightarrow (\mathcal{C}(\SS(\AA_{j}))_*,\partial_*)
$$
define an exact sequence of algebraic complexes 
\begin{equation}\label{eq:succ1}
0 \rightarrow (\mathcal{C}(\SS(\AA_{j-1}))_*,\partial_*) \rightarrow (\mathcal{C}(\SS(\AA_j))_*,\partial_*) \rightarrow (F^1(\SS(\AA_j)),\partial_*)\rightarrow 0
\end{equation}
where $F^1(\SS(\AA_j))$ denotes the quotient complex $\mathcal{C}(\SS(\AA_{j-1})) / \mathcal{C}(\SS(\AA_{j}))$ with the induced boundary map. It is generated by $k$-uples $(H_1, \ldots H_k) \in \nbc_k(\AA_j)$ such that the last hyperplane $H_k$ belongs to $\AA_j \setminus \AA_{j-1}$, $k=1,\ldots ,j$, that is $F^1(\SS(\AA_j))$ has no $0$-cells and $1$-cells are of the form $(H)$, $H \in \AA_j \setminus \AA_{j-1}$. The short exact sequence in (\ref{eq:succ1}) defines the long one in homology
\begin{equation}\label{eq:longseq}
\begin{split}
\ldots \longrightarrow H_q&(\mathcal{C}(\SS(\AA_{j-1})),\mathbb{L}) \longrightarrow H_q(\mathcal{C}(\SS(\AA_{j})),\mathbb{L}) \longrightarrow \\
& \longrightarrow H_q(F^1(\SS(\AA_j)),\mathbb{L}) \longrightarrow H_{q-1}(\mathcal{C}(\SS(\AA_{j-1})),\mathbb{L}) \longrightarrow \ldots \quad.
\end{split}
\end{equation}

\subsection{First (co)homology group} In \cite{gaiffi2009morse}, the authors gave a simplified formula for the boundary map described in \cite{salvettisette} in the case of line arrangements. In particular, they described a simplified formula for the second boundary map $\partial_2$. Indeed, in order to study the first (co)homology group of the minimal complex $(\mathcal{C}(\SS)_*,\partial_*)$, we can simply refer to the line arrangement obtained intersecting the arrangement $\AA$ with the $2$-dimensional space $V_2$ in the fixed flag of spaces $\{V_k\}_{k=0,\ldots,d}$. From now on we will refer to this line arrangement.

More in details, the formula in \cite{gaiffi2009morse} is defined as follows. For a given $2$-critical cell  $[C \prec p]$, define the following subset of lines
$$
S(p):=\{H \in \AA \mid p \in H\},
$$
and denote with $H_{S(p)}$ (resp. $H^{S(p)}$) the line in $S(p)$ with minimum (resp. maximum) index. There are two lines in $S(p)$ (lines in the section $V_2$) which bound $C$. Denote by $H_C$ (resp. $H^C$) the one which has minimum (resp. maximum) index. Let $\Cone(p)$ be the closed cone bounded by $H_{S(p)}$ and $H^{S(p)}$ having vertex $p$ and whose intersection with $V_1$ is bounded. Then define
\begin{equation*}
\begin{split}
& U(C):=\{H_i \in S(p) \mid i \geq \mbox{ index of } H^C \}\\
& L(C):=\{H_i \in S(p) \mid i \leq \mbox{ index of } H_C \}\\
\end{split}
\end{equation*}
where lines in $U(C)$ are the lines of $S(p)$ which do not separate the chambers $C_0$ and $C$ while the lines in $L(C)$ are the ones which separate $C_0$ from $C$.

Consider a line $V_{p}$ that passes trough the points $V_0$ and $p$. This line intersects all the lines $H\in \AA$  in points $P_H$ and, if $\theta(P_H)$ is the length of the segment $V_0P_H$, then define the set 
$$U(p):=\{H \in \AA \mid \theta(P_H)> \theta(p)\}.$$

In addition, in \cite{gaiffi2009morse} the authors described a rank $1$ local system on the line arrangements as follows. If $\Lbb$ is an abelian local system over $\mathcal{M}(\AA)$ and $V_0$ the basepoint, they fixed a positive orientation and associated an element $t_H \in$ Aut$(\Lbb)$ to each elementary loop turning around the complexified line $H_{\C}$, for all $H \in \AA$. They got a homomorphism
$$\Z[\pi_1(\mathcal{M}(\AA)]\rightarrow\Z[H_1(\mathcal{M}(\AA)]\rightarrow\Z[t_H^{±1}]_{H \in \AA} \subseteq \mbox{End}(\Lbb).$$

With this system of coefficients they described the formula which computes the first local system (co)homology group for the complement of an arrangement $\AA$, as follows.
\begin{equation}\label{eq:boundary}
\begin{split}
&\partial_2(l.e_{[C \prec p]})=\sum_{\substack{ |p_j|\in S(p)}}\bigg(\prod_{\substack{i<j~s.t.\\H_i\in U(p)}}t_{H_i}\bigg)\bigg[\prod_{\substack{i~s.t.\\H_i\in[C\rightarrow|p_j|)}}t_{H_i}-\prod_{\substack{i<j~s.t.\\H_i\in S(p)}}t_{H_i}\bigg](l).e_{[C_{j-1}\prec p_j]}\\
&+\sum_{\substack{|p_j|\in U(p)\\p_j\subset \Cone(p)}}\bigg(\prod_{\substack{i<j~s.t.\\H_i\in U(p)}}t_{H_i}\bigg)\bigg(1-\prod_{\substack{i<j~s.t.\\ H_i\in L(C)}}t_{H_i}\bigg)\bigg(\prod_{\substack{i<j~s.t.\\H_i\in U(C)}}t_{H_i}-\prod_{\substack{i~s.t.\\H_i\in U(C)}}t_{H_i}\bigg)(l).e_{[C_{j-1}\prec p_j]},
\end{split}
\end{equation}
where $l \in \Lbb$, $[C_{j-1} \prec p_j]$ are the $1$-critical cells, $p_j \in H_j$ and $[C\rightarrow |p_j|)$ are the subsets of $S(p)$ defined by
\begin{enumerate}
\item[i)] $[C\rightarrow|p_j|):=\{H_k\in U(C)~|~k<j\}$ if $|p_j|\in U(C)$;
\item[ii)] $[C\rightarrow|p_j|):=\{H_k\in S(p)~|~k<j\}\cup U(C)$ if $|p_j|\in L(C)$.
\end{enumerate}

Furthermore, because the only critical $0$-cell is $[C_0\prec C_0]$, the boundary map $\partial_1$ can be easily computed
\begin{equation*}
\partial_1(l.e_{[C_{i-1}\prec p_i]})=(1-t_{H_i})(l).e_{[C_0\prec C_0]}.
\end{equation*}
 
Notice that computing the (co)homology of the Milnor fiber with integer coefficients is equivalent to set, in the above boundary map, all the elements $t_H \in$ Aut$(\Lbb)$ equal to the same $t\in$ Aut$(\Lbb)$. \\

\subsection{Milnor fiber of supersolvable arrangements} Let us now consider a supersolvable arrangement $\AA$ endowed with the recursive special ordering $\pol$ and a rank $1$ local system $\Lbb_t$ obtained setting $t_H=t$ for all $H \in \AA$. 
We can define the augmented inclusion map
\begin{equation} \label{eq:mappaau}
\begin{split}
i[1]\colon \bigoplus_{H \in \AA_j \setminus \AA_{j-1}}\mathcal{C}(\SS(\AA_1))[1] &\rightarrow  F^1(\SS(\AA_j))\\
\end{split}
\end{equation}
that sends each copy of the only $0$-cell (respectively $1$-cell) in $\mathcal{C}_0(\SS(\AA_1))$ (respectively $\mathcal{C}_1(\SS(\AA_1))$) in the $1$-cell $(H)$ (respectively $2$-cell $(H_1,H)$), $H \in \AA_j \setminus \AA_{j-1}$.

We get the following Theorem.

\begin{Theorem}\label{thm:bprin}With the previous notations, if all hyperplanes in the difference $\AA_j \setminus \AA_{j-1}$ intersect generically the hyperplane in $\AA_1$ the map $i[1]$ is a map of algebraic complexes.
\end{Theorem}
\begin{proof} It is enough to show that $i[1]$ commutes with the boundary operator. For any $H \in \AA_j \setminus \AA_{j-1}$, we have
$$i[1]\circ\partial_1((H_1))=i[1]((1-t)()=(1-t)H$$
where $()$ stands for the $0$-cell of $\mathcal{C}(\SS(\AA_1))$ and, using formula in (\ref{eq:boundary}),
$$\partial_2\circ i[1]((H_1))=\partial_2((H_1,H))=(1-t)H$$
\end{proof}

Let us now consider the case of coefficients in the rank 1 local system $\Lbb_{t,\Q}$ obtained by replacing $\Z[t,t^{-1}]$ with  $\Q[t,t^{-1}]$. In this case there is no $\Z$ torsion and, as  $\Q[t,t^{-1}]$ is principal ideal domain, it is known that
$$H_1(\mathcal{C}(\SS(\AA_1)),\Lbb_{t,\Q}) \simeq \bigoplus_{n \mid \sharp \AA}[\Q[t,t^{-1}]/\varphi_n]^{\beta_n}$$
where each $\beta_n \geq 0$ and $\varphi_n$ are the cyclotomic polynomial of degree $n$.
An immediate consequence of Theorem \ref{thm:bprin} is that, as $H_0(\mathcal{C}(\SS(\AA_1)),\Lbb_{t,\Q}) \simeq \Q[t,t^{-1}]/(1-t) \simeq \Q$ then  $H_1(F^1(\SS(\AA_j)) \simeq \Q^{b_j}$, $b_j \leq \sharp (\AA_j \setminus \AA_{j-1})$. This comes directly from remarks on the long exact sequence in homology induced by inclusion $i[1]$. Indeed the quotient complex $$F^1(\SS(\AA_j)) /  \bigoplus_{H \in \AA_j \setminus \AA_{j-1}}\mathcal{C}(\SS(\AA_1))[1]$$ has no $1$-cells and the las part of the long exact sequence in homology becomes

\begin{equation*}
\ldots \longrightarrow H_0(\mathcal{C}(\SS(\AA_{1})),\mathbb{L}_{t,\Q}) \longrightarrow H_1(F^1(\SS(\AA_j)),\mathbb{L}_{t,\Q}) \longrightarrow 0 \quad .
\end{equation*}

Moreover, if the cardinalities of $\AA_{j-1}$ and $\AA_j$ are coprime, the long exact sequence in homology induced by the exact sequence in (\ref{eq:succ1}) splits into short exact sequences of the form
$$0 \rightarrow H_q(\mathcal{C}(\SS(\AA_j)),\Lbb_t)\rightarrow H_q(F^1(\SS(\AA_j)),\Lbb_t)\rightarrow H_{q-1}(\mathcal{C}(\SS(\AA_{j-1})),\Lbb_t)\rightarrow 0$$
and the following Theorem holds.

\begin{Theorem}\label{th:bounmain} Let $\AA =\AA_d \supset \ldots \supset \AA_1$ be a supersolvable arrangement in $\R^d$. If it exists an index  $1 \leq j \leq d$ such that the cardinalities of $\AA_{j-1}$ and $\AA_j$ are coprime and all hyperplanes in $\AA_j \setminus \AA_{j-1}$ intersect generically the hyperplane in $\AA_1$, then  $H_1(\mathcal{C}(\SS(\AA_j)),\Lbb_{t,\Q}) \simeq \Q^{b_j}$, $b_j \leq \sharp (\AA_j \setminus \AA_{j-1})$.
\end{Theorem}

If $F(\AA)$ is the Milnor fiber associated to $\AA$, then $H_1(\mathcal{C}(\SS(\AA_j)),\Lbb_{t,\Q} \simeq H_1(\mathcal{M}(\AA_j),\Lbb_{t,\Q}) \simeq H_1(F(\AA),\Q)$ and Theorem \ref{th:bounintro} is proved.

Remark that any line arrangement $\AA$ that admits a filtration $\AA_1=\{H_1\} \subset \AA_2 \subset \AA_3$ verifying condition (2) of Definition \ref{supers} can be regarded as the deconing of a supersolvable arrangement $c\AA$ and Theorem \ref{th:bounmain} applies.

Notice that the condition of $H \in \AA_j \setminus \AA_{j-1}$ and $H_1 \in \AA_1$ be normal crossing hyperplanes 
is necessary in order $i[1]$ to be a map of algebraic complexes. But in order to prove Theorem \ref{th:bounmain} one just needs the existence of a map of algebraic complexes with the same image of $i[1]$. Hence a natural question is whether a less restrictive condition would be enough to prove the existence of such a map in order to generalize Theorem \ref{th:bounmain}. 

\section{Braid arrangement}\label{sec:braid}

In this Section, we describe the isomorphism between the symmetric group and the Orlik-Solomon algebra for the braid arrangement $$\AA=\{H_{ij}=\{x_i=x_j\},1\le i<j\le n+1\}.$$

We indicate simply by $A_n$ the symmetric group on $n+1$ elements, acting by permutations of the coordinates. Then $\AA = \AA(A_n)$ is the braid arrangement and $\SS(A_n)$ is the associated CW-complex.

Notice that $\AA$ is a supersolvable arrangement with filtration given by $\AA_1=\{H_{12}\}$ and $\AA_j\setminus\AA_{j-1}=\{H_{1j+1},\dots, H_{jj+1}\}$, for $j=2,\dots,n$.

In \cite{salvettisette}, the authors gave a tableaux description of $\SS(A_n)$ and constructed singular tableaux, that is tableaux corresponding to critical faces.

\subsection{Tableaux description of $\SS(A_n)$ and singular tableux}\label{subsec:tab}
Given a system of coordinates in $\R^{n+1}$ it is possible to describe $\SS(A_n)$ through certain tableaux as follow.

Every $k$-cell $[C \prec F]$ is represented by a tableau with $n+1$ boxes and  $n+1-k$ rows (aligned on the left), filled with all the integers in $\{1,\ldots, n~+~1~\}.$ There is no monotony condition on the lengths of the rows. One has:
\medskip

\ni - $(x_1,\ldots, x_{n+1})$ is a point in $F$ iff:\\

$1.$ $i$ and $j$ belong to the same row iff $x_i=x_j$,

$2.$ $i$ belongs to a row less than the one containing $j$ iff $x_i < x_j$;\\

\ni - the chamber $C$ belongs to the half-space $x_i < x_j$ iff:\\

$1.$ either the row which contains $i$ is less than the one containing $j$ or

$2.$ $i$ and $j$ belong to the same row and the column which contains $i$ is less than the one containing $j$.
\medskip

Notice that the geometrical action of $A_n$ on the stratification induces a natural action on the complex $\SS(A_n),$ which, in terms of
tableaux, is given by a left action of $A_n$: $\sigma. \ T$ is the tableau with the same shape as $T,$ and with entries permuted through $\sigma.$

Denote by $\Tbf(A_n)$ the set of ``row-standard" tableaux, i.e. with entries increasing along each
row. Each face in the stratification $\FF(A_n)$ corresponds to an equivalence class of tableaux, where the equivalence is up to row preserving
permutations. Let $\Tbf^{\mathbf{k}}(A_n)$ be the set of tableaux of dimension $k$
(briefly, $k$-tableaux), i.e. tableaux with exactly $n+1-k$ rows.  Moreover, write $T \prec T'$ iff $F \prec F'$, where the
tableaux $T$ and $T'$ correspond respectively to $F$ and $F'$. 

Define the following operations between tableaux
\begin{enumerate}
\item $T * T^{\prime}$ is the new tableau obtained by attaching vertically $T^{\prime}$ below $T$;
\item $T *_i h$ is the tableau obtained by attaching the one-box tableau with entry $h$ to the $i$-th row of $T$;
\item $T^{op}$ is the tableau obtained from $T$ by reversing the row order.
Notice that $(T * T^{\prime})^{op}=T^{\prime op} * T^{op}$.
\end{enumerate}

Fix $k$ integers $1 < j_1 < \cdots < j_k \leq n+1$ and, for any $1 \leq h \leq k+1$, let $T_h$ be the $0$-tableau (= one-column tableau) with entries $J_h=\{j_{h-1}+1, \dots , j_h-1\}$ in the natural order (set $j_0=0,\ j_{k+1}=k+2$).
Then, for any suitable choice of integers $i_1,\ldots ,i_k$ define a $k$-tableau
\begin{equation}\label{tabdipi}
T^k=((\cdots ((((T_1^{op}*_{i_1} j_1) * T_2)^{op}*_{i_2} j_2) * T_3)^{op} \cdots)^{op}*_{i_k} j_k) * T_{k+1}.
\end{equation}

In \cite{salvettisette}, the authors proved that there exists a system of polar coordinates, generic with respect to $\AA(A_n)$, such that a $k$-facet $p$ is critical if and only if the tableau $T_p$ which represents $p$ is of the form in \eqref{tabdipi}. Moreover the induced order $\pol$ between $k$-critical facets $p$ equals the order between $k$-\textit{critical} tableaux defined as the order induced by lexicographic order between sequences of pairs $((j_1,i_1),\ldots,(j_k,i_k))$, where $(j_t,i_t)<(j'_{t},i'_t)$ if and only if either $j_t<j'_t$ or $j_t=j'_t$ and $i_t>i'_t$.  While a $k$-critical tableau is smaller than an $h$-critical one, with $k \neq h$, if and only if $k <h$. This ordering is special recursive ordering. As no confusion can arise, we will denote this order among critical tableaux by $\pol$.

We will now describe how to attach to each critical cell a $0$-tableau and so to built a bijection between the Orlik-Solomon algebra of $\AA$ and the symmetric group.

\subsection{ Non broken circuits of the symmetric group}
Fix $p$ a $k$-critical face, $[p]=([H_{i_1}]_p, \dots, [H_{i_k}]_p)$. Consider then a $k$-critical cell $[C\prec p]$ with $f_k([C\prec p])=(H'_1,\dots, H'_k)$, where $H'_j\in[H_{i_j}]_p$. Because we are considering the braid arrangement, for any $H\in\AA$ we can write $H=H_{(s,t)}$ for some $1\le s<t\le n+1$. Hence we can write $$([H_{i_1}]_p, \dots, [H_{i_k}]_p)=([H_{(s_{1},t_{1})}]_p, \dots, [H_{(s_{k},t_{k})}]_p)$$ and $$(H'_1,\dots, H'_k)=(H_{(s'_1,t_1)}, \dots, H_{(s'_k,t_k)}).$$ 

Let $T_p$ be the tableau attached to the $k$-critical face $p$. Then $T_p$ is of the form described in (\ref{tabdipi}). Suppose that $T_p$ has rows $(T_p)_1,\dots,(T_p)_{n+1-k}$.

We will describe how to attach to $[C\prec p]$ one and only one $0$-tableaux $T_{[C\prec p]}$ starting from the tableau $T_p$. 
\begin{Definition} Given a tableau $T$, define the map  $$\mathfrak{r}_T\colon \{1,\dots, n+1\}\to\{1,\dots,n+1\}$$ such that $\mathfrak{r}_T(j)$ is the row of $T$ where $j$ is. If $T=T_F$ for a face $F$, we will simply write $\mathfrak{r}_F$ instead of $\mathfrak{r}_{T_F}$.
\end{Definition}

\begin{Remark} By construction, if $H_{s,t}$ is an hyperplane such that  $H_{s,t} \supset p$ then $s$ and $t$ lies in the same row of the tableau $T_p$, that is $\mathfrak{r}_p(s)=\mathfrak{r}_p(t)$, while if  $H_{s,t} \not \supset p$ then $ \mathfrak{r}_p(s)\ne \mathfrak{r}_p(t)$.
Moreover, if $C$ is a chamber, that is a $0$-codimensional facet, then it is represented by a column tableau $T_C$. Let $C_0$ be the base chamber corresponding to the tableau with entry $i$ in the $i$-th row. By construction, if the hyperplane $H_{s,t}$, $s<t$, separates the chambers $C$ and $C_0$ then  $\mathfrak{r}_C(s) > \mathfrak{r}_C(t)$, while $\mathfrak{r}_C(s) < \mathfrak{r}_C(t)$ otherwise. It is easy to see that if $H_{s,t}$ is a wall of $C$, then  $| \mathfrak{r}_C(s) - \mathfrak{r}_C(t) |=1$, that is $s$ and $t$ belong to consecutive rows. 
\end{Remark}
\begin{Definition}\label{defconditionstabl} Consider $f_k([C\prec p])=(H_{(s'_1,t_1)}, \dots, H_{(s'_k,t_k)})$. Then we define an order $<_{[C\prec p]}$ on the set of integers $\{s'_1,\ldots,s'_k,t_1,\ldots,t_k\}$ as follows
\begin{enumerate} 
\item[i)] if $(-1)^{k-j}=1$, then $t_j<_{[C\prec p]}s'_j$;
\item[ii)] if $(-1)^{k-j}=-1$, then $s'_j<_{[C\prec p]}t_j$;
\item[iii)] $t_k<_{[C\prec p]}s'_k$ are consecutive numbers in the order $<_{[C\prec p]}$.
\end{enumerate}
\end{Definition}
\begin{Proposition}\label{prop:ord} $<_{[C\prec p]}$ is a total order on the entries of $(T_p)_i$, for all $i=1,\dots, n+1-k$.
\end{Proposition}
\begin{proof} 
We will prove the statement using induction on $k$. Suppose $k=1$. Then the statement is obvious as $T_p$ is a tableau with $n-1$ rows of length one and one row of length two with entries $\{s_1,t_1\}$ and, by definition, $t_1<_{[C\prec p]}~s_1$.

Suppose now $k>1$ and let $(H_{(s'_1,t_1)}, \dots, H_{(s'_{k-1},t_{k-1})})$ be the $k-1$-uple obtained removing the last entry of $f_k([C\prec p])=(H_{(s'_1,t_1)}, \dots, H_{(s'_k,t_k)})$. By Lemma \ref{lem:sottofaccia} and Corollary \ref{correspcelluple}, there exists a unique $(k-1)$-critical cell $[C'\prec \breve{p}]$ attached to it. Hence we can consider the tableau $T_{\breve{p}}$ and we know by induction that $<_{[C'\prec \breve{p}]}$ is a total order on the rows of $T_{\breve{p}}$. $T_p \prec T_{\breve{p}}$ differ only by the rows that contains $s'_k$ and $t_k$ as $\mathfrak{r}_{\breve{p}}(s'_k) \ne \mathfrak{r}_{\breve{p}}(t_k)$ while $\mathfrak{r}_{p}(s'_k)=\mathfrak{r}_{p}(t_k)$. As a consequence we just need to prove that $<_{[C\prec p]}$ is a total order on the row $\mathfrak{r}_{p}(s'_k)=\mathfrak{r}_{p}(t_k)$. In all the other rows of $T_p$ the order is given by
$$b<_{[C\prec p]}a \quad \mbox{ if and only if} \quad a<_{[C'\prec \breve{p}]}b.$$ 
Notice that, by tableaux definition,  $\sharp(T_{\breve{p}})_{\mathfrak{r}_{\breve{p}}(t_k)}=1$ and hence 
$(T_{p})_{\mathfrak{r}_p(s'_k)}=(T_{\breve{p}})_{\mathfrak{r}_{\breve{p}}(s'_k)}\cup\{t_k\}$. 
By condition iii) in Definition \ref{defconditionstabl} $t_k<_{[C\prec p]}s'_k$ are consecutive numbers in the order $<_{[C\prec p]}$ then, if  $a\in(T_{\breve{p}})_{\mathfrak{r}_{\breve{p}}(s'_k)}\setminus\{s'_k\}$, $a<_{[C'\prec \breve{p}]}s'_k$ implies $a<_{[C\prec p]}t_k$ and, similarly, $s'_k<_{[C'\prec \breve{p}]}a$ implies $t_k<_{[C\prec p]}a$, i.e. $<_{[C\prec p]}$ is a total order.
\end{proof}
For all $i=1,\dots, n+1-k$, denote by $T_{[C\prec p]}^i$ the column tableau obtained transposing the $i$-th row 
$(T_p)_i$ of the tableau $T_p$ with entries ordered from upper to bottom by $<_{[C\prec p]}$. Define
$$T_{[C\prec p]}:=T_{[C\prec p]}^1\ast\cdots\ast T_{[C\prec p]}^{n+1-k}.$$ 
The following Theorem holds.

\begin{Theorem} The map $\mathfrak{T}\colon\textup{Crit}(\SS(A_n)) \longrightarrow \Tbf^{\mathbf{0}}(A_{n})$ defined by $\mathfrak{T}([C\prec p])=T_{[C\prec p]}$ is a bijection.
\end{Theorem}
By the Definition \ref{defconditionstabl}, we get that the bijection $\mathfrak{T}$ factorizes through the bijection $f$. That is, there exists a bijection  $$g\colon \nbc(A_n) \longrightarrow \Tbf^{\mathbf{0}}(A_{n})$$ between non broken circuit basis of the Orlik-Solomon algebra $A(A_n)$ and $0$-tableaux such that $\mathfrak{T}= g\circ f$. Moreover, as tableaux in $\Tbf^{\mathbf{0}}(A_{n})$ naturally correspond to permutations in the symmetric group $S_{n+1}$ on one hand and to chambers $C \simeq [C \prec C]$ of the braid arrangement $\AA(A_n)$ on the other hand, if  we consider  the map $\eta$ described in \eqref{eq:eta}, we get the following diagram of maps
\begin{equation} \begin{array}{ccc}\label{eq:maps}
\textup{Crit}(\SS(A_n)) &\xrightarrow{f} &\nbc(A_n)  \\
\eta \downarrow & & g \downarrow \\
\ch(A_n) & \xrightarrow{\varphi} & \Tbf^{\mathbf{0}}(A_{n})\simeq S_{n+1}
\end{array}
\end{equation}
where $\varphi$ is the bijection described in Section \ref{subsec:tab}. The following Theorem holds.

\begin{Theorem} The diagram in \eqref{eq:maps} is a commutative diagram of bijective maps.
\end{Theorem}

The bijection $g$ retrieves, for Coxeter group $A_n$, the map constructed by Barcelo and Goupil in \cite{BarceloGupil}. They proved the existence of a bijection between non broken circuit basis of the Orlik-Solomon algebra and elements of the group for all real reflection groups. However, they only proved the existence of the inverse map $g^{-1}$ by induction. In our construction, it is possible to get the inverse map by direct computation as follows.

Let $\mathcal{T}(A_n)$ be the set of all critical tableaux computed as in equation (\ref{tabdipi}).
Given a permutation $w \in A_n$ and the associated tableau $T_w$, then $T_w=T_{C'}$ for a chamber $C'=\op_p(C)=\eta([C \prec p]) $ where $p$ is the smallest critical face in the order $\pol$ such that $C'\prec p$ and hence
$$T_p=\min_{\pol}\{T \in \mathcal{T}(A_n) \mid T_{C'}\prec T\}.$$

Then, given $T_{C'}$ and $T_p$ it is possible to retrieve the \textit{nbc}-uple of hyperplanes associated to $w$ via Definition \ref{defconditionstabl}.

An interesting question is whether this construction can be extended to other reflection groups.

\bigskip

\bigskip


\noindent
\textbf{Acknowledgments:} 
During the preparation of this paper, the second author was partially supported by JSPS Postdoctoral Fellowship For Foreign Researchers, the Grant-in-Aid (No. 23224001 (S)) for Scientific Research, JSPS.

\bibliography{bibliothesis}{}
\bibliographystyle{plain}
\end{document}